\documentclass[12pt]{amsart}
   \usepackage{amsmath,amsthm}
   \usepackage{amsfonts}   
   \usepackage{amssymb}    
   \usepackage{url}

\advance \topmargin by -\headheight
\advance \topmargin by -\headsep
      
\evensidemargin \oddsidemargin
\usepackage[all]{xy}

\newtheorem{thm}{Theorem}[section]
\newtheorem{prop}[thm]{Proposition}
\newtheorem{lem}[thm]{Lemma}
\newtheorem{cor}[thm]{Corollary}

\newtheorem*{thmp_1_fact}{Proposition \ref{p_1_fact}}
\newtheorem*{thmgamma_non_triv}{Corollary \ref{gamma_non_triv_gen}}

\theoremstyle{remark}
\newtheorem{rem}[thm]{Remark}

\numberwithin{equation}{section}

\title{On foliated characteristic classes of transversely symplectic foliations}
\author{Jonathan Bowden}
\address{Mathematisches Institut, Ludwig-Maximilians-Universit\"at, Theresienstr. 39, 80333 M\"unchen, Germany}
\curraddr{Mathematisches Institut, Universit\"at Augsburg, Universit\"atsstr. 14, 86159 Augsburg, Germany}
\email{jonathan.bowden@math.uni-augsburg.de}
\date{\today}
\begin{document}
\begin{abstract}
Kotschick and Morita recently discovered factorisations of characteristic classes of transversely symplectic foliations that yield new characteristic classes in foliated cohomology. We describe an alternative construction of such factorisations and construct examples of topologically trivial foliated $\mathbb{R}^{2n}$-bundles for which these characteristic classes are non-trivial. This shows that the foliated cohomology classes of Kotschick and Morita carry information that is not merely topological.
\end{abstract}

\maketitle
\section{Introduction}
In \cite{KM3} Kotschick and Morita defined foliated characteristic classes of transversely symplectic foliations in terms of factorisations of certain combinations of Pontryagin classes. Motivated by this we show that similar factorisations also exist for foliations that admit smooth invariant transverse volume forms.
\begin{thmp_1_fact}
Let $\mathcal{F}$ be a foliation of codimension $2n$ on a manifold $M$. Assume that $\mathcal{F}$ has a smooth invariant transverse volume form $\Omega$ and let $P$ be any polynomial of total degree $4n$ in the Pontryagin classes of the normal bundle. Then there is a factorisation
\[P(F^{\nabla}) = \Omega \wedge \gamma_P\]
for a well-defined foliated class $\gamma_P \in H^{2n}_{\mathcal{F}}(M)$.
\end{thmp_1_fact}
\noindent  The factorisations in \cite{KM3} were obtained by calculations involving the Gelfand-Fuks cohomology of the Lie algebra of formal Hamiltonian vector fields. We argue more directly using Chern-Weil theory for Bott connections, which in particular gives an alternate approach to the factorisations of Kotschick and Morita for foliations with transverse symplectic structures (cf.\ Propositions \ref{p_1_fact_symp} and \ref{equiv_gamma}). 

The foliated cohomology classes $\gamma_P$ determine genuine characteristic class of leaves by means of restriction. By using results on normal forms for germs of Hamiltonian diffeomorphisms we show the existence of foliations for which the restrictions of the classes $\gamma_{p^n_1}$ are non-trivial.
\begin{thmgamma_non_triv}
There exist topologically trivial foliated $\mathbb{R}^{2n}$-bundles such that the restriction of $ \gamma_{n} = \gamma_{p^n_1}$ to some closed leaf $L$ is non-trivial. In particular, the class $\gamma_{n}$ is non-trivial in the foliated cohomology of the total space.
\end{thmgamma_non_triv}
\noindent These examples show that the classes $\gamma_P$ carry information that is sensitive to the geometry of the foliation and that they do not merely depend on the homotopy class of the underlying distribution as is the case for the ordinary Pontryagin classes. In the special case of codimension $2$ foliations the characteristic class given by restricting $\gamma_{1}$ to a leaf $L$ is trivial under the assumption that $L$ has no linear holonomy (cf.\ Remark \ref{rem_lin_holonomy}). Thus $\gamma_{1}$ is an obstruction to leaves having trivial linear holonomy, even if their normal bundles are \emph{a priori} topologically trivial. 

\subsection*{Acknowledgments:}
The main motivation and impetus for this article were given by the stimulating questions and continued encouragement of Prof.\ D.\ Kotschick during the course of the author's doctoral studies, from which the results of this article are taken. 

We also thank the referee for the numerous suggestions that significantly improved the exposition of this article. The financial support of the Deutsche Forschungsgemeinschaft is also gratefully acknowledged.

\subsection*{Notation and Conventions:}
All manifolds and foliations will be assumed to be oriented. Unless otherwise stated all homology groups will be taken with integral coefficients. Finally a topological group will be decorated with a $\delta$ when it is to be considered as a discrete group.
\section{Factorisations of Pontryagin classes}\label{Fact_Pontryagin}
We first recall the definitions of Bott connections and foliated cohomology.
For a foliation $\mathcal{F}$ we let $I^*(\mathcal{F})$ denote the ideal of forms that vanish on $\mathcal{F}$. The Frobenius Theorem implies that $I^*(\mathcal{F})$ is closed under the exterior differential. The foliated cohomology is then defined as the cohomology of this quotient complex:
\[H^*_{\mathcal{F}}(M) = H^*(\Omega^*(M) / I^*(\mathcal{F})).\]
A \emph{Bott connection} on the normal bundle $\nu_\mathcal{F} = TM / T \mathcal{F}$ is a connection $\nabla$ such that for $X \in \Gamma(T \mathcal{F})$ and  $Y \in \Gamma(\nu_\mathcal{F})$
\[\nabla_X \thinspace Y = [X, \tilde{Y}], \]
where $\tilde{Y}$ denotes any lift of $Y$ to $TM$. The most important properties of Bott connections are that they are flat when restricted to the leaves of $\mathcal{F}$ and that they are canonically defined along any leaf by the formula above (cf.\ \cite{Bott}). Conversely, to define a Bott connection for a given foliation one chooses a splitting
\[TM \cong T \mathcal{F} \oplus \nu_{\mathcal{F}}.\]
If $X = X_{\mathcal{F}} + X_{\nu}$ is the decomposition of a vector $X$ with respect to this splitting, then after choosing a connection $\overline{\nabla}$ on $\nu_{\mathcal{F}}$, one may define a Bott connection $\nabla$ as follows:
\[\nabla_X \thinspace Y = [X_{\mathcal{F}}, \tilde{Y}] + \overline{\nabla}_{X_\nu} \thinspace Y.\]
We also note the following standard lemma on differential forms (cf.\ \cite{Pit}).
\begin{lem}\label{int_formula}
Let $\alpha \in \Omega^k(M \times [0,1])$ be a $k$-form and let $\iota_0, \iota_1$ be the inclusions of $M \times \{ 0 \}$ resp. $M \times \{1 \}$ in $M \times [0,1]$ and $\pi$ the projection to $M$. Then the following relation holds
\[\pi_!  \thinspace d \alpha -  d \thinspace \pi_! \alpha = \iota_1^* \alpha - \iota_0^* \alpha.\]
\end{lem}
\noindent We now come the main result of this section.
\begin{prop}\label{p_1_fact}
Let $\mathcal{F}$ be a foliation of codimension $2n$ on a manifold $M$. Assume that $\mathcal{F}$ has a smooth invariant transverse volume form $\Omega$ and let $P$ be any polynomial of total degree $4n$ in the Pontryagin classes of the normal bundle. Then there is a factorisation
\[P(F^{\nabla}) = \Omega \wedge \gamma_P\]
for a well-defined foliated class $\gamma_P \in H^{2n}_{\mathcal{F}}(M)$.
\end{prop}
\begin{proof}
Let $\nabla$ be a Bott connection on the normal bundle $\nu_{\mathcal{F}}$ of $\mathcal{F}$. Since a Bott connection is flat along leaves, the components $F^{\nabla}_{ij}$ of the curvature matrix vanish on $\mathcal{F}$. We choose a local basis $\theta_1,..., \theta_{2n}$ of $I^1(\mathcal{F})$ such that 
\[ \theta_1 \wedge \theta_2... \wedge \theta_{2n} = \Omega.\]
With respect to this basis the curvature forms $F^{\nabla}_{ij} \in I^2(\mathcal{F})$ can locally be expressed as
\[F^{\nabla}_{ij} = \sum_{k = 1}^{2n} \theta_k \wedge \alpha_k.\]
Since the Chern-Weil representative for $P$ is given by a symmetric polynomial of degree $2n$ in the entries of $F^{\nabla}$, the following holds locally:
\[P(F^{\nabla}) = \theta_1 \wedge \theta_2... \wedge \theta_{2n} \wedge \gamma_P = \Omega \wedge \gamma_P.\]
Let $Ann(\Omega)$ be the subbundle of $2n$-forms annihilated by $\Omega$ under the wedge product and let $Ann(\Omega)^{\perp}$ be a choice of a complement so that $$Ann(\Omega) \oplus  Ann(\Omega)^{\perp}  \cong \Lambda^{2n}(M).$$
Then on the level of forms the equation
\[P(F^{\nabla}) = \Omega \wedge \gamma_P\]
has a unique global solution $\gamma_P \in \Gamma(Ann(\Omega)^{\perp})$, which is well-defined modulo elements in $\Gamma(Ann(\Omega)) = I^{2n}(\mathcal{F})$. Next, since $\Omega$ and $P(F^{\nabla}) $ are closed we compute
\[0 = d(\Omega \wedge \gamma_P) = \Omega \wedge d \gamma_P\]
and $d \gamma_P \in I^*(\mathcal{F})$. Thus we have a well-defined class $[\gamma_P] \in H^{2n}_{\mathcal{F}}(M)$.

We must now show that the class we obtain in foliated cohomology does not depend on the choice of Bott connection. Let $\nabla^0, \nabla^1$ be two Bott connections on $\nu_{\mathcal{F}}$ and let $\pi$ denote the projection $M \times [0,1] \to M$. We then define a connection on $\pi^* \nu_{\mathcal{F}}$ by setting
\[\nabla = t \thinspace \pi^*\nabla^1 + (1-t) \pi^*\nabla^0.\]
This is then a Bott connection for the foliation $\pi^* \mathcal{F}$ that is obtained as the preimage of $\mathcal{F}$ under the projection $\pi$. Now $\pi^* \Omega$ is a defining form for $\pi^* \mathcal{F}$ and, as above, after the choice of a splitting $\Lambda^{2n}(M \times [0,1]) \cong  Ann(\pi^* \Omega) \oplus Ann(\pi^*\Omega)^{\perp}$, there is a unique form $\gamma_P \in \Gamma(Ann(\pi^* \Omega)^{\perp})$ so that
\[ P(F^{\nabla}) = \pi^* \Omega \wedge \gamma_P.\]
Since the form $P(F^{\nabla}) $ is closed, Lemma \ref{int_formula} yields
\begin{align*} 
- d \thinspace \pi_! P(F^{\nabla}) & =  \iota_1^*  P(F^{\nabla})- \iota_0^*  P(F^{\nabla})\\
& = P(F^{\nabla_1}) - P(F^{\nabla_0}).
\end{align*}
Hence, one has
\[ \Omega \wedge (- d \thinspace \pi_!  \gamma_P) = \Omega \wedge \gamma_P^1 - \Omega \wedge \gamma_P^0\]
or equivalently
\[\gamma_P^1 - \gamma_P^0 \equiv - d (\pi_! \gamma_P) \text{ mod } I^*(\mathcal{F}).\]
Thus $[\gamma_P^1] = [\gamma_P^0]$ as classes in $H^{2n}_{\mathcal{F}}(M)$.
\end{proof}

If one further assumes that the foliation $\mathcal{F}$ is transversely symplectic with defining form $\omega$, then one obtains a similar factorisation for any polynomial of the form $\omega^k \wedge P(F^{\nabla})$, where $P$ is a polynomial in the Pontryagin classes of total degree $4(n - k)$. By applying the proof of Proposition \ref{p_1_fact} \emph{mutatis mutandis} we obtain the following, which was proven in \cite{KM3} by using Gelfand-Fuks cohomology. 
\begin{prop}[\cite{KM3}, Th.\ 4]\label{p_1_fact_symp}
Let $\mathcal{F}$ be a transversely symplectic foliation of codimension $2n$ on a manifold $M$ with defining form $\omega$ and let $P$ be any polynomial of total degree $4(n - k)$ in the Pontryagin classes of the normal bundle. Then there is a factorisation
\[\omega^k \wedge P(F^{\nabla}) = \omega^{n} \wedge \gamma_P\]
for a well-defined foliated class $\gamma_P \in H^{2n}_{\mathcal{F}}(M)$.
\end{prop}
For certain polynomials it is easy to show that the class $\omega^k \wedge P(F^{\nabla}) = \omega^{n} \wedge \gamma_P$ is non-trivial and, hence, that the class $\gamma_P$ is non-trivial in foliated cohomology. In particular, this was shown in \cite{KM3} for polynomials of the form $P = p_1^{q-k}$. This leaves open the question of the non-triviality of the classes $\gamma_P$ for foliations whose normal bundles are topologically trivial, which we shall consider in greater detail below.

\section{Gelfand-Fuks cohomology}\label{GF_cohomology}
In this section we clarify the relationship between the foliated classes given in Proposition \ref{p_1_fact_symp} and those defined in \cite{KM3}. In particular, we show that both descriptions agree under the assumption that the normal bundle of $\mathcal{F}$ is trivial.

We begin by recalling the construction of Gelfand-Fuks cohomology for the Lie algebra $\mathfrak{ham}_{2n}$ of formal Hamiltonian vector fields (cf.\ \cite{GKF}, \cite{KM3}). We let $H^{2n}_{\mathbb{R}}$ denote the standard $Sp(2n,\mathbb{R})$-representation and consider the ring of polynomials on $H^{2n}_{\mathbb{R}}$ with trivial constant term
\[ \mathbb{R}[x_1,..., x_n, y_1,...,y_n] / \mathbb{R}= \oplus_{k = 1}^{\infty} S^k H^{2n}_{\mathbb{R}},\]
where $S^k H^{2n}_{\mathbb{R}}$ is identified with the space of homogeneous polynomials of degree $k$. This space can be identified with $\mathfrak{ham}_{2n}$ by associating to each polynomial its formal Hamiltonian vector field:
$$ \mathbb{R}[x_1,..., x_n, y_1,...,y_n] / \mathbb{R} \ni H \longmapsto \sum_{i=1}^{n} \left( \frac{\partial H}{\partial x_i}\frac{\partial}{\partial y_i} -  \frac{\partial H}{\partial y_i}\frac{\partial}{\partial x_i} \right) \in \mathfrak{ham}_{2n}.$$
Under this correspondence the bracket of formal vector fields corresponds to the Poisson bracket on polynomials:
\[ \{f,g\} = \sum_{i=1}^{n} \left( \frac{\partial f}{\partial x_i}\frac{\partial g}{\partial y_i} -  \frac{\partial f}{\partial y_i}\frac{\partial g}{\partial x_i} \right). \]
Note that if $f \in S^k H^{2n}_{\mathbb{R}}, g \in  S^l H^{2n}_{\mathbb{R}}$ then the bracket $\{f,g\}$ lies in $S^{k+l-2} H^{2n}_{\mathbb{R}}$. The Gelfand-Fuks cochain complex is then
\[C^*_{GF}(\mathfrak{ham}_{2n}) = \Lambda^* (\mathfrak{ham}^*_{2n}),\]
where the differential on $n$-forms is defined in terms of the Poisson bracket using the usual formula for Lie algebra cohomology:
\[ (d \alpha)(f_0,...,f_n) = \sum_{i < j} (-1)^{i+j} \alpha(\{f_i,f_j\},f_0,...,\widehat{f_i},...,\widehat{f_j},...,f_n), \ \forall \  f_0,...,f_n \in \mathfrak{ham}_{2n}.\]
The Lie algebra $\mathfrak{ham}_{2n}$ has a natural filtration
\[\mathfrak{ham}_{2n} \supset \mathfrak{ham}_{2n}^0 \supset  \mathfrak{ham}_{2n}^1 ... \supset \mathfrak{ham}_{2n}^k \supset ... \  ,\]
where $\mathfrak{ham}_{2n}^k$ denotes those polynomials that are trivial up to order $k + 1$. The subalgebras $\mathfrak{ham}_{2n}^k$ are actually ideals in $\mathfrak{ham}_{2n}^0$ and the quotient $\mathfrak{g}^k_{ham} = \mathfrak{ham}_{2n}^0 / \mathfrak{ham}_{2n}^k$ is the Lie algebra of the group of $k$-jets of Hamiltonian maps that fix $0$, which will be denoted by $J^k_{ham}$. Moreover, the Lie algebra $\mathfrak{sp}_{2n} = S^2 H^{2n}_{\mathbb{R}}$ embeds naturally in $\mathfrak{ham}_{2n}$. Finally recall that for a subalgebra $\mathfrak{g} \subset \mathfrak{ham}^k_{2n}$, the relative complex is by definition the sucomplex of $\mathfrak{g}$-basic forms:
\[C^*_{GF}(\mathfrak{ham}^k_{2n},\mathfrak{g}) = \{\alpha \in C^*_{GF}(\mathfrak{ham}^k_{2n})  \ | \ \iota_{X} \alpha = 0 = \iota_X d \alpha, \ \forall \   X \in \mathfrak{g} \}.\]

We shall next recall the construction of the natural map from Gelfand-Fuks cohomology to foliated cohomology following \cite{BH}:
$$H^*_{GF}(\mathfrak{ham}_{2n}^0, \mathfrak{sp}_{2n}) \stackrel{\Phi} \longrightarrow H^*_{\mathcal{F}}(M).$$
Let $\Gamma^{ham}_{2n}$ denote the Lie pseudogroup of Hamiltonian diffeomorphisms of open sets in $\mathbb{R}^{2n}$ and let $P^k(\Gamma^{ham}_{2n})$ denote the principal $J^k_{ham}$-bundle of $k$-jets of elements in $\Gamma^{ham}_{2n}$ at $0$, which is naturally a bundle over $\mathbb{R}^{2n}$. Note that the pseudogroup $\Gamma^{ham}_{2n}$ acts transitively on $P^k(\Gamma^{ham}_{2n})$ from the left. More generally, if $M$ carries a transversely symplectic foliation we let 
$$P_{ham}^k(\mathcal{F}) \stackrel{\pi_k} \longrightarrow M$$
be the principal $J^k_{ham}$-bundle whose fibre at $x$ consists of the $k$-jets of local $\mathcal{F}$-projections that preserve the transverse structure and map $x$ to $0$.

Any element $\gamma$ in $C^*_{GF}(\mathfrak{ham}_{2n})$ determines a $\Gamma^{ham}_{2n}$-equivariant differential form $\Phi^{loc}_{\gamma}$ on $P^k(\Gamma^{ham}_{2n})$ for some sufficiently large $k$. Since the bundle $P_{ham}^k(\mathcal{F})$ is locally the pullback of the bundle $P^k(\Gamma^{ham}_{2n})$ under an $\mathcal{F}$-projection and the form $\Phi^{loc}_{\gamma}$ is $\Gamma^{ham}_{2n}$-equivariant, the pullbacks of these local forms glue together to give a well-defined form $\Phi_{\gamma}$ on the total space.

Let $\gamma$ be a representative of a given class in $H^m_{GF}(\mathfrak{ham}_{2n}^0, \mathfrak{sp}_{2n})$ and consider the pullback of $\gamma$ to the chain complex $C_{GF}^m(\mathfrak{ham}_{2n}, \mathfrak{sp}_{2n})$ induced by the projection 
$$H^{2n}_{\mathbb{R}} \longrightarrow  \mathfrak{ham}_{2n} \longrightarrow \mathfrak{ham}_{2n}^0.$$
Note that this map is \emph{not} a map of Lie algebras so that the pullback of $\gamma$ may not be closed. However, since $\mathfrak{ham}^0_{2n}$ is a subalgebra, it follows that it is closed modulo elements in the differential ideal generated by $(H^{2n}_{\mathbb{R}})^*$.

We then apply the construction described above to obtain a differential form $\Phi_{\gamma}$ on the $J^k_{ham}$-principal bundle $P_{ham}^k(\mathcal{F})$ over $M$. Since $\gamma$ was assumed to be $\mathfrak{sp}_{2n}$-basic, the form $\Phi_{\gamma}$ descends to a form $\hat{\Phi}_{\gamma}$ on the quotient bundle
 $$\hat{P}_{ham}^k(\mathcal{F}) = P_{ham}^k(\mathcal{F}) / Sp(2n,\mathbb{R}) \stackrel{\hat{\pi}_k}\longrightarrow M.$$
Moreover, the elements associated to the ideal generated by $(H^{2n}_{\mathbb{R}})^*$ vanish on the pullback foliation $\pi_k^*\mathcal{F}$ on $P_{ham}^k(\mathcal{F})$. Thus after quotienting out by $Sp(2n,\mathbb{R})$ we deduce that $\hat{\Phi}_{\gamma}$ defines an element in the foliated cohomology of the pullback foliation $\hat{\pi}_k^* \mathcal{F}$ on $\hat{P}_{ham}^k(\mathcal{F})$. Since the fibres of the quotient bundle are contractible, the map $\hat{\pi}_k$ then induces an isomorphism
$$H_{\hat{\pi}_k^* \mathcal{F}}^*(P_{ham}^k(\mathcal{F})) \stackrel{\cong} \longrightarrow H_{\mathcal{F}}^k(M),$$
and we define $\Phi(\gamma) = (\hat{\pi}^*_k)^{-1}(\hat{\Phi}_{\gamma})$. If the bundle $P_{ham}^k(\mathcal{F})$ is itself trivial and $s$ is any section, then $\Phi(\gamma) = s^*\Phi_{\gamma}$ as classes in $H_{\mathcal{F}}^k(M)$.

In the case of trivial normal bundles, it follows easily that the foliated classes $\gamma^{KM}_P$ defined via factorisations in Gelfand-Fuks cohomology in \cite{KM3} coincide with the classes $\gamma_P$ as given by Proposition \ref{p_1_fact_symp}.
\begin{prop}\label{equiv_gamma}
Let $\mathcal{F}$ be a transversely symplectic foliation of codimension $2n$ and assume that its normal bundle is trivial. Then for any polynomial $P$ as in Proposition \ref{p_1_fact_symp} we have $ \gamma^{KM}_P = \gamma_P$.
\end{prop}
\begin{proof}
Let $P$ be a polynomial as in Proposition \ref{p_1_fact_symp}. Then the class $\gamma^{KM}_P$ determines a factorisation 
$$\omega^k \wedge P(\Omega^i_j) =  \omega^n \wedge \gamma^{KM}_P$$
in Gelfand-Fuks cohomology, where $\Omega^i_j$ denotes the curvature of the universl Bott connection (\cite{KM3}, p.\ 13 ff.). Under the map $H^*_{GF}(\mathfrak{ham}_{2n}^0, \mathfrak{sp}_{2n}) \to H^*_{\mathcal{F}}(M)$ described above, this decomposition then becomes a factorisation of $\omega^k \wedge P(\Omega^i_j)$ for the universal Bott connection on $P_{ham}^3(\mathcal{F})$. More specifically, if $\omega$ is a transverse symplectic form for $\mathcal{F}$ and $\pi_3$ is the bundle projection of $P_{ham}^3(\mathcal{F})$, then $\pi_3^* \thinspace \omega$ is a transverse symplectic form for $\pi_3^{*} \mathcal{F}$ and
\[\omega^k \wedge P(\Omega^i_j) =  (\pi^*_3 \thinspace \omega)^{n} \wedge \gamma^{KM}_P(\pi_3^{*} \mathcal{F}).\]
By assumption the bundle $P_{ham}^3(\mathcal{F}) \rightarrow M$ has a section $s$. Applying $s^*$ to both sides of this equation, we conclude that
$$\omega^n \wedge \gamma_P(\mathcal{F}) = \omega^n \wedge \gamma^{KM}_P(\mathcal{F}) \Longrightarrow \gamma^{KM}_P(\mathcal{F}) = \gamma_P(\mathcal{F}) \text{ in } H^*_{\mathcal{F}}(M). \qedhere$$
\end{proof}

\section{Leaf invariants and Reinhart's construction}\label{Reinhart}
As in the previous section, we let $P_{ham}^k(\mathcal{F})$ be the principal $J^k_{ham}$-bundle of $k$-jets of local $\mathcal{F}$-projections that preserve the transverse symplectic structure of a foliation $\mathcal{F}$. If $L$ is a leaf of $\mathcal{F}$, then the restriction of the bundle $P_{ham}^k(\mathcal{F})$ to $L$ is a flat principal $J^k_{ham}$-bundle. An analogous construction to that given in Section \ref{GF_cohomology}, associates to any class in $H_{GF}^*(\mathfrak{ham}_{2n}^0, \mathfrak{sp}_{2n})$ a class in $H^*(L,\mathbb{R})$. Here the space $P^k(\Gamma^{ham}_{2n})$ is replaced by the Lie group $J^k_{ham}$ and the forms $\Phi_{\gamma}$ are just left invariant forms associated to elements in $\Lambda^*(\mathfrak{g}_{ham}^{k})^* $. Moreover, one has the following commutative diagram, where the arrow on the left is induced by the quotient map and the arrow on the right is given by restriction
\[ \xymatrix{H^*_{GF}(\mathfrak{ham}_{2n}^0, \mathfrak{sp}_{2n}) \ar[r]^-{\Phi} & H_{\mathcal{F}}^*(M) \ar[d]  \\
H^*(\mathfrak{g}_{ham}^{k}, \mathfrak{sp}_{2n}) \ar[u] \ar[r]^-{\Phi} & H^*(L,\mathbb{R}).}\]
The characteristic classes of the individual leaves that are defined via this method are called leaf invariants and were first considered by Reinhart (see \cite{STi}). Since leaf invariants are in particular characteristic classes of flat principal $J^k_{ham}$-bundles they determine elements in group cohomology and we obtain the following \emph{Reinhart map}:
$$H^*(\mathfrak{g}_{ham}^{k}, \mathfrak{sp}_{2n}) \longrightarrow H^*((J^k_{ham})_{\delta},\mathbb{R}).$$

\section{Non-triviality of leafwise characteristic classes}
The aim of this section is to prove the non-triviality of the leaf invariant associated to the class $\gamma_1=\gamma^{KM}_{p_1}$ and we begin with a short summary of the steps involved in proof.
\medskip

\textbf{Step 1:} Show that $\gamma_1$ is non-trivial in $H^*((J^m_{ham})_{\delta},\mathbb{R})$ for all $m$.

\medskip

\textbf{Step 2:} Show non-triviality in $H^*(\mathcal{G}^k_{ham},\mathbb{R})$ for germs of differentiability class $C^k$.

\medskip

\textbf{Step 3:} Show non-triviality for flat $\mathbb{R}^2$-bundles with holonomy of class $C^k$.

\medskip

We first turn to step 1. The class $\gamma_1$ in $H^2_{GF}(\mathfrak{ham}_{2}^0, \mathfrak{sp}_2)$ is defined via the explicit cocycle representative (\cite{KM3}, p.\ 11)
\[(x^3 \wedge y^3)^* - \frac{1}{3}(x^2y \wedge xy^2)^* \in \Lambda^2(S^3H^2_{\mathbb{R}})^*.\]
Here we have identified $S^3H^2_{\mathbb{R}}$ with the space of homogeneous degree $3$ polynomials and the vectors $(x^3 \wedge y^3)^*, (x^2y \wedge xy^2)^*$ are elements of the basis of $\Lambda^2(S^3H^2_{\mathbb{R}})^*$ induced by the standard basis
$$\{(x^3)^*,(x^2y)^*,(xy^2)^*,(y^3)^*\} \subset (S^3H^2_{\mathbb{R}})^*$$
 under wedge products. We claim that the cohomology class $\gamma_1 \in H^2((J^2_{ham})_{\delta}, \mathbb{R})$ given by Reinhart's construction is non-trivial. For let $K_2 \cong S^3H^2_{\mathbb{R}} \cong\mathbb{R}^4$ be the kernel of the natural map
\[J^2_{ham} \to J^1_{ham}.\]
If $\mathfrak{k}_{2}$ denotes the Lie algebra of $K_2$, we have the following commutative diagram:
\[ \xymatrix{H^*(\mathfrak{g}_{ham}^{2}, \mathfrak{sp}_{2}) \ar[r] \ar[d] & H^2((J^2_{ham})_{\delta}, \mathbb{R}) \ar[d]  \\
H^*(\mathfrak{k}^*_{2}) = \Lambda^2 (\mathbb{R}^4)^* \ar[r] & H^2(\mathbb{R}^4_{\delta}, \mathbb{R}).}\]
The image of $\gamma_1$ in $\Lambda^2 (\mathbb{R}^4)^*$ is non-zero and the Reinhart map is injective for any vector space. We conclude that $\gamma_1$ is non-trivial in $H^2((J^2_{ham})_{\delta}, \mathbb{R})$.
\begin{rem}\label{Euler_independent}
Note that $\gamma_1$ and the Euler class are linearly independent, since the Euler class restricts trivially to the kernel of the map $J^2_{ham} \to J^1_{ham}$.
\end{rem}
\noindent The following proposition then completes Step 1.
\begin{prop}\label{inj_jets}
For all $m \geq 3$ the natural map $J^m_{ham} \to J^2_{ham}$ induces an injection
$$H^2((J^{2}_{ham})_{\delta}) \to H^2((J^{m}_{ham})_{\delta}).$$
\end{prop}
\begin{proof}
We first note that the map $J^m_{ham} \to J^2_{ham}$ factors as
\[J^{m}_{ham} \to J^{m-1}_{ham} \to ... \to J^2_{ham}.\]
Thus it suffices to show that the map $J^{m}_{ham} \to J^{m-1}_{ham}$ induces an injection on cohomology for all $m \geq 3$. The kernel of this map can be identified with the additive group of homogeneous polynomials $S^{m+1} H^{2}_{\mathbb{R}}$ of degree $m+1$ and the conjugation action of $J^{m-1}_{ham}$ is given by precomposition. Furthermore, the conjugation action factors through the natural map
\[J^{m-1}_{ham} \to J^{1}_{ham} = SL(2,\mathbb{R}).\]
We consider the five-term exact sequence associated to the extension  
\[1 \to S^{m+1} H^{2}_{\mathbb{R}} \to J^{m}_{ham} \to J^{m-1}_{ham} \to 1.\]
Then since the $SL(2,\mathbb{R})$-representation $S^{m+1} H^{2}_{\mathbb{R}}$ is irreducible, the group of co-invariants $H_1((S^{m+1} H^{2}_{\mathbb{R}})_{\delta})_{J^{m-1}_{ham}}$ is trivial. Thus by exactness, the injectivity of the map in second cohomology follows.
\end{proof}
\begin{rem}\label{linear_holonomy}
An explicit computation shows that
\[\frac{1}{9}d(x^2y^2)^* = (x^3 \wedge y^3) - \frac{1}{3}(x^2y \wedge xy^2)^* \text{ in } C^*_{GF}(\mathfrak{ham}_{2}^1),\]
which in particular implies that $\gamma_1$ is a cocycle. This, however, does \emph{not} hold in the $SL(2,\mathbb{R})$-equivariant subcomplex
$$C^*_{GF}(\mathfrak{ham}_{2}^0,\mathfrak{sp}_2) = C^*_{GF}(\mathfrak{ham}_{2}^1)^{SL(2,\mathbb{R})} \subset C^*_{GF}(\mathfrak{ham}_{2}^1),$$
since $(x^2y^2)^*$ is not $SL(2,\mathbb{R})$-equivariant. The above computation also holds in the Lie algebra cohomology $H^*(\mathfrak{g}^m_{ham,Id})$ of the group $J^m_{ham, Id}$ of jets with trivial linear part for all $m \geq 3$. In particular, if $\gamma_1$ is non-trivial for a particular $J^{m}_{ham}$-bundle, then the image of holonomy map cannot lie in the kernel of the map $J^{m}_{ham} \to J^{1}_{ham} = SL(2,\mathbb{R})$.
\end{rem}
In order to extend our results from jets to actual germs we shall need a normal form theorem for Hamiltonian germs. This is provided by the following result of Banyaga, de la Llave and Wayne. 
\begin{thm}[\cite{BLW}, Th.\ 1.1]\label{lin_BLW}
Let $\phi$ be the germ of a $C^r$-Hamiltonian diffeomorphism fixing the origin, whose $(k-2)$-jet is linear and hyperbolic. Then there exists a $C^l$-Hamiltonian germ $\psi$ with $D_0 \psi = Id$ so that
\[\psi^{-1} \phi \psi = D_0 \phi,\]
provided that $r > 2k + 4$ and $ 1 \leq l \leq k A - B$, where $A,B$ are constants depending only on the eigenvalues of $D_0 \phi$.
\end{thm}
\noindent This is the main tool needed to show that the class $\gamma_1$ induces a non-trivial class in the cohomology of the group $\mathcal{G}^k_{ham}$ of Hamiltonian germs that have differentiability class $C^k$, thus concluding Step 2.
\begin{thm}\label{germ_case}
The class $\gamma_1$ is non-trivial in $H^2(\mathcal{G}^k_{ham})$ for all $2 \leq k < \infty$. 
\end{thm}
\begin{proof}
By Proposition \ref{inj_jets} the image of $\gamma_1$ in $H^2((J^m_{ham})_{\delta}, \mathbb{R})$ is non-trivial for all $m$. We let $\sigma$ be an integral homology class in $H_2((J^m_{ham})_{\delta})$ that pairs non-trivially with $\gamma_1$. Such a class is equivalent to a representation of some surface group $\pi_1(\Sigma_g) \to J^m_{ham}$. We let $a_i,b_i$ be standard generators for $\pi_1(\Sigma_g)$ and let $\alpha_i, \beta_i$ denote their images in $J^m_{ham}$. Then if we consider the $m$-jets $\alpha_i,  \beta_i$ as (smooth) germs in the natural way, we see that the $m$-jet of
\[\phi = \prod_{i=1}^g [\alpha_i,\beta_i]\]
is the identity. Thus after multiplying $\phi$  with a hyperbolic element $M \in SL(2, \mathbb{R})$, Theorem \ref{lin_BLW} implies that we may linearise the resulting germ by a Hamiltonian diffeomorphism germ $\psi$ of class $C^l$ with $D_0\psi = Id$. Since the value of $l$ grows linearly with $m$, after taking $m$ large enough we may assume that $l = k$ and, hence, the following holds in $\mathcal{G}^k_{ham}$
\[\psi \prod_{i=1}^g [\alpha_i,\beta_i] M \psi^{-1} = M \Longleftrightarrow \prod_{i=1}^g [\alpha_i, \beta_i] [M,\psi^{-1}] = 1.\]
We let $\sigma'$ denote the homology class in $H_2(\mathcal{G}^k_{ham})$ associated to the representation of the surface group $\pi_1(\Sigma_{g + 1})$ given by the above relation. The image of $\sigma'$ in $H_2((J^2_{ham})_{\delta})$ then decomposes as $\sigma + \tau$, where $\sigma$ is our original class and $\tau$ is the class associated to the representation of the fundamental group of the 2-torus defined by the $2$-jets of the elements $\psi^{-1}$ and $M$. Since $D_0\psi = Id$, the conjugation action of $M$ on $\psi^{-1}$ is by precomposition in $S^{3} H^2_{\mathbb{R}}$, which can only be trivial if $\psi^{-1} = Id$. Thus we conclude that $\tau = 0$ and hence $\gamma_1(\sigma') = \gamma_1(\sigma) \neq 0$.
\end{proof}
To complete the third and final step we must extend classes given by germs to classes defined by actual $C^k$-symplectomorphisms. We begin with the following lemma.
\begin{lem}\label{Thurston_argument}
Let $\phi$ be an element in the group $Ham_c^k(\mathbb{R}^{2} \setminus \{ 0 \})$ that consists of Hamiltonian diffeomorphisms of class $C^k$ with compact support. Then $\phi$ can be written as a product of commutators of elements in $Symp^k(\mathbb{R}^{2} \setminus \{ 0 \})$ which are the identity near the origin.
\end{lem}
\begin{proof}
First of all since $\phi$ can be connected to the identity by a path of Hamiltonian diffeomorphisms with compact support, the standard fragmentation argument given in the smooth case applies and we may assume that $\phi$ is a product of elements with support in a ball (cf.\ \cite{Ban}, p.\ 110). Thus it suffices to consider the case where $\phi$ has support in a ball $B$ that is disjoint from a closed neighbourhood $N$ about the origin. Let $\gamma \in Symp^k(\mathbb{R}^{2} \setminus 0)$ be such that the sets $\gamma^n(B)$ are pairwise disjoint and are disjoint from $N$ for all $n \geq 0$. By choosing $\gamma$ appropriately we may further assume that it fixes a neighbourhood of the origin. We then define
$$\psi = \prod_{n = 0}^{\infty} \gamma^{n} \circ \phi \circ \gamma^{-n} \in Symp^k(\mathbb{R}^{2} \setminus \{ 0 \})$$
and one computes that $ \phi = [\psi,\gamma]$.
\end{proof}
We will also need an explicit description of the Euler class of an $\mathbb{R}^2$-bundle in terms of its holonomy representation. Using the fact that the inclusion of the compactly supported symplectomorphism group $Symp_c^k(\mathbb{R}^{2} \setminus \{ 0 \})$ in $Diff_c^k(\mathbb{R}^{2} \setminus \{ 0 \})$ is a weak homotopy equivalence, the following is a straight forward application of the methods discussed in \cite{Bow2}, and we leave the details to the reader.
\begin{prop}\label{Fukui_Symp}
Let $\alpha$ be the natural homomorphism from $Symp_c^k(\mathbb{R}^2 \setminus \{ 0 \})$ to the compactly supported mapping class group $MCG_c(\mathbb{R}^2 \setminus \{ 0 \}) \cong \mathbb{Z}$ and consider the following extension of groups
\[ 1 \to Symp_c^k(\mathbb{R}^2 \setminus \{ 0 \}) \to Symp_c^k(\mathbb{R}^{2},0) \to \mathcal{G}^k_{ham} \to 1.\]
Further let $a_i,b_i$ be the standard generators of $\pi_1(\Sigma_g)$ and consider a representation $$\pi_1(\Sigma_g) \stackrel{\rho} \longrightarrow \mathcal{G}^k_{ham}.$$ If $\tilde{\alpha_i}, \tilde{\beta_i}$ are any lifts of $\rho(a_i),\rho(b_i)$ to $Symp_c^k(\mathbb{R}^{2},0)$, then the Euler number of the associated $\mathbb{R}^2$-bundle $E$ is given by the following expression:
\[e(E) = \pm \alpha(\prod_{i=1}^g [\tilde{\alpha_i}, \tilde{\beta_i}]).\]
\end{prop}
As a final preliminary we note that the ordinary Flux homomorphism on the identity component $Symp^k_{c,0}(\mathbb{R}^2 \setminus \{0\})$ extends to the group $Symp_c^k(\mathbb{R}^{2},0)$. As usual for any path joining $Id$ to $\phi \in Symp^k_{c,0}(\mathbb{R}^2 \setminus \{0\})$ and $X_t = \frac{\partial}{\partial t}\phi_t$ the Flux homomorphism is defined as
\[Flux(\phi) = \int_0^1\iota_{X_t} \omega \in H^1_c(\mathbb{R}^2 \setminus \{0\},\mathbb{R}).\]
Although one usually only considers Flux homomorphisms on groups of smooth symplectomorphisms, the familiar properties of $Flux$ also hold in the case of symplectomorphisms that are only of class $C^k$ and the proofs given in e.g.\ \cite{McS} for the smooth case hold \emph{verbatim} for $C^k$-symplectomorphisms. In particular, if $\omega = -d\lambda$, then $Flux(\phi) = [\lambda - \phi^*\lambda]$ and $Ham^k_c(\mathbb{R}^2 \setminus \{0\})$ agrees with the kernel of $Flux$. Moreover, there is a unique compactly supported function $H_{\phi}$ on $\mathbb{R}^2$ so that $\lambda - \phi^*\lambda = dH_{\phi}$ and under the natural isomorphism $H^1_c(\mathbb{R}^2 \setminus \{0\},\mathbb{R}) \to \mathbb{R}$ given by integrating over a sufficiently long interval $[0,R]$ we see that
\[Flux(\phi) =\lim_{R \to \infty} \int_{[0,R]} \lambda - \phi^*\lambda = \lim_{R \to \infty} (H_{\phi}(R) - H_{\phi}(0)) = - H_{\phi}(0). \]
This definition in terms of primitives then extends immediately to $Symp_c^k(\mathbb{R}^{2},0)$ and we denote this extension by $Flux^{(\mathbb{R}^2,0)}$.

We are now ready to prove that there are foliations on $\mathbb{R}^2$-bundles with closed leaves so that the restriction of $\gamma_1$ is non-trivial.
\begin{thm}\label{gamma_non_triv}
For any $2 \leq k < \infty$ there exist topologically trivial foliated $\mathbb{R}^2$-bundles with holonomy in $Symp^k(\mathbb{R}^2,0)$ for which the restriction of the characteristic class $\gamma_1$ to the leaf corresponding to the origin is non-trivial. In particular, the class $\gamma_1$ is non-trivial in the foliated cohomology of the total space.
\end{thm}
\begin{proof}
By Theorem \ref{germ_case} there exists a representation of some surface group $\pi_1(\Sigma_g) \to \mathcal{G}^k_{ham}$ for which $\gamma_1$ is non-trivial. We let $a_i,b_i$ be the standard generators of $\pi_1(\Sigma_g)$ and let $\alpha_i, \beta_i$ denote the images of these generators in $\mathcal{G}^k_{ham}$. We further let $\tilde{\alpha_i}, \tilde{\beta_i}$ be representatives of these germs that have compact support in $\mathbb{R}^2$. Then by construction
\[ \phi = \prod_{i=1}^g [\tilde{\alpha_i}, \tilde{\beta_i}] \]
is an element in $Symp_{c}^k(\mathbb{R}^2 \setminus \{ 0 \})$. As the Euler class can be chosen to be trivial (see Remark \ref{Euler_independent}), it follows from Proposition \ref{Fukui_Symp} that $\phi$ lies in the identity component of $Symp_c^k(\mathbb{R}^2 \setminus \{ 0 \})$.

Moreover, $\phi$ lies in $Ham_c^k(\mathbb{R}^2 \setminus \{ 0 \})$ since
\[Flux^{\mathbb{R}^2 \setminus \{ 0 \}}(\phi) = Flux^{(\mathbb{R}^{2},0)}(\phi) = Flux^{(\mathbb{R}^{2},0)}(\prod_{i=1}^g [\tilde{\alpha_i}, \tilde{\beta_i}]) = 0.\]
Thus by Lemma \ref{Thurston_argument} we may write $\phi^{-1}$ as a product of $N$ commutators of elements that have support disjoint from $0$. We may then define a representation $\pi_1(\Sigma_{g + N}) \to Symp^k(\mathbb{R}^2,0)$ on which the class $\gamma_1$ is non-trivial and by construction the associated foliated bundle is topologically trivial.
\end{proof}
\begin{rem}\label{rem_lin_holonomy}
Although the bundles in Theorem \ref{gamma_non_triv} are topologically trivial, by Remark \ref{linear_holonomy} we know that the central leaf must have linear holonomy. Thus the class $\gamma_1$ may be viewed as an obstruction to having trivial linear holonomy, even for leaves with topologically trivial normal bundles.
\end{rem}
\noindent By considering products of the foliations given in Theorem \ref{gamma_non_triv} we obtain the following generalisation as a corollary.
\begin{cor}\label{gamma_non_triv_gen}
There exist topologically trivial foliated $\mathbb{R}^{2n}$-bundles such that the restriction of $ \gamma_{n} = \gamma_{p^n_1}$ to some closed leaf $L$ is non-trivial. In particular, the class $\gamma_{n}$ is non-trivial in the foliated cohomology of the total space.
\end{cor}
\begin{proof}
Let $E$ be as in Theorem \ref{gamma_non_triv} and consider the $n$-fold product $E_n = E \times ... \times E$ with the product foliation $\mathcal{F}_n = \mathcal{F} \times ... \times \mathcal{F}$. This is then transversely symplectic with $\omega_n = \pi_1^* \omega + ... +  \pi_n^* \omega$, where $\pi_i$ denotes the projection to the $i$-th factor. For any Bott connection on $E$ the connection $\nabla_n = \pi_1^*\nabla \oplus ... \oplus \pi_n^*\nabla$ is Bott for $\mathcal{F}_n$ and
\[p_1(F^{\nabla_n}) = \pi_1^* p_1(F^{\nabla}) + ... +  \pi_n^* p_1(F^{\nabla}) = \pi_1^* \omega \wedge \pi_1^* \gamma_1 + ... +  \pi_n^* \omega \wedge \pi_n^* \gamma_n\]
and it follows by dimension considerations that
\begin{align*}
p^n_1(F^{\nabla_n}) & = (\pi_1^* \omega \wedge \pi_1^* \gamma_1 + ... +  \pi_n^* \omega \wedge \pi_n^* \gamma_1)^n
= (\pi_1^* \omega + ... +  \pi_n^* \omega)^n \wedge (\pi_1^* \gamma_1 + ... +  \pi_n^* \gamma_1)^n\\ 
& = \omega_n^n \wedge (\pi_1^* \gamma_1 + ... +  \pi_n^* \gamma_1)^n = \Omega \wedge \gamma_{n}(\mathcal{F}_n).
\end{align*}
Hence the restriction of $\gamma_{n}(\mathcal{F}_n)$ to $L_n = L \times ... \times L$ is $(\pi_1^* \gamma_1|_L + ... +  \pi_n^* \gamma_1|_L)^n$, which is non-trivial by the K\"unneth formula and the fact that $\gamma_1|_L$ is non-trivial by assumption.

\end{proof}
Corollary \ref{gamma_non_triv_gen} gives a generalisation of the non-triviality statement of Theorem 4 in \cite{KM3} to foliations with trivial normal bundles. This is however gained at the expense of considering foliations that are only of class $C^k$ on non-compact manifolds. We would further hope that it is possible to show similar results for transversely symplectic foliations on compact manifolds and for foliations that are smooth. The former would require an analogue of Banyaga's results for symplectomorphisms of class $C^k$, whereas the latter would probably require the development of new techniques since an extension of the results of \cite{BLW} to the smooth case seems unlikely.

\end{document}